\documentclass[11pt]{article}
\usepackage[dvips]{graphicx}
\usepackage{amssymb,color}
\usepackage{epic}
\usepackage{pstricks}
\usepackage{pst-node}
\usepackage{pstricks-add}

\def\BBox{\kern  -0.2cm\hbox{\vrule width 0.2cm height 0.2cm}}

\newtheorem{theorem}{Theorem}[section]
\newtheorem{lemma}[theorem]{Lemma}

\newtheorem{conjecture}[theorem]{Conjecture}

\usepackage{color}

\title{Mixed Cages}
\author{G. Araujo-Pardo, C. Hern\'andez-Cruz,
J.J. Montellano-Ballesteros\\
\thanks{ \footnotesize{\em Email addresses:} ~ garaujo@matem.unam.mx (G. Araujo-Pardo),
~ tetragrammatos@gmail.com  (C. Hern\'andez-Cruz),\, \, ~ 
juancho@matem.unam.mx (J.J. Montellano-Ballesteros)}
{\footnotesize Instituto de Matem\'{a}ticas, Universidad Nacional Aut\'{o}noma de M\'{e}xico,} \\
{\footnotesize M\'{e}xico D. F., M\'exico }
}

\begin{document}
\maketitle

\begin{abstract}

We introduce the notion of a $[z, r; g]$-mixed cage.   A $[z, r; g]$-mixed cage is a mixed graph $G$, $z$-regular by arcs, $r$-regular by edges, with girth $g$ and minimum order.  In this paper we prove the existence of $[z, r ;g]$-mixed cages and exhibit families of mixed cages for some specific values.   We also give lower and upper bounds for some choices of $z, r$ and $g$.  In particular we present the first results on $[z,r;g]$- mixed cages for $z=1$ and any $r\geq 1$ and $g\geq 3$, and for any $z\geq 1$, $r=1$ and $g=4$.
\vspace{1 mm}

\noindent{\bf Key words:}  Mixed graph, Mixed cage, Mixed Moore graphs.

\end{abstract}

\section{Introduction}
\textcolor{green}{}

In this paper we consider graphs which are finite and mixed, i.e., they may contain (directed) arcs as well as (undirected) edges. We allow multiple edges and arcs. 

 A {\it{mixed regular graph}} is a simple and finite graph $G$, such that each vertex $v$ of $G$ is the start-point of $z$ arcs, the end-point of $z$ arcs and is incident with $r$ edges.  $z$ is the directed degree and $r$ is the undirected degree of a vertex $v$, and we set $d=r+z$; $d$ is the degree of $v$ ($r$, $z$, and $d$ are independent of the choice of the vertex). We will consider walks of the form $(v_0, \dots, v_n)$, where for each $i \in \{0, \dots, n-1\}$, $v_i v_{i+1}$ is an edge of $G$ or $(v_i, v_{i+1})$ is an arc of $G$.   In other words, we will consider walks that could contain edges and arcs, provided that all the arcs are traversed in the same direction.   With this kind of walks (and
hence cycles), we define the girth of $G$ as the length of the shortest cycle of $G$. If $G$ has girth equal to $g$, we say that $G$ is a $[z,r;g]$-{\em{mixed graph}} of directed degree $z$, undirected degree $r$ and girth $g$.  


   A $[z, r; g]$-{\em mixed cage} is a $[z, r; g]$-mixed graph of minimum order.  We are interested in the problem of finding $[z, r; g]$-mixed cages for different values of $z, r$ and $g$.
   
   If $z = 0$, then we are talking about the well known {\it{Cage Problem}} that has been widely studied since cages were introduced by Tutte \cite{T47} in 1947 and after Erd\"os and Sachs proved, in 1963, their existence in  \cite{ES63}. A complete survey about this topic and its relevance can be found in \cite{EJ08}. 
   
There exists a natural lower bound for the order of a cage, which will be used througout this paper. It is called {\it {Moore's lower bound}}, denoted by $n_0(r,g)$, and is obtained by counting the vertices of a tree, ${\cal{T}}_{(g-1)/2}$, rooted on a vertex and with radius $(g-1)/2$, if $g$ is odd; or the vertices of a ``double-tree'' rooted at an edge (that is, two different ${\cal{T}}_{(g-3)/2}$ trees rooted each one at the vertices incident with an edge) if $g$ is even (see \cite{EJ08}):

  \begin{equation}\label{lowercages} n_0(r,g) = \left\{ \begin{array}{ll} 1 + r + r(r-1) + \cdots
+ r(r-1)^{(g-3)/2} &\mbox{ if $g$ is odd};\\
2(1 +(r-1) +
\cdots + (r-1)^{g/2-1})&\mbox{ if $g$ is
even}.\end{array}\right.\end{equation}

A known problem related to the {\it {Cage Problem}} is the {\it{$(\Delta,d)$-Problem}}, which consists of finding graphs of maximum order, with maximum degree $\Delta$ (they could be regular of degree $\Delta$), and fixed diameter $d$ (maximum distance between each pair of vertices). The upper bound of this parameter is called the {\it{Moore bound}} and coincide with the lower bound for the cages of odd girth $g=2d+1$ (where $d$ is the diameter of the graph). The graphs that attain this upper bound are called {\it{Moore Graphs}}, and this is the reason for calling the cages that attain the lower bound $n_0(k,g)$ {\it{Moore Cages}}. We recall this because, one of our motivations for the study of ``Mixed Cages'' is related with the study of ``Mixed Moore Graphs'' (see \cite{ABMM16}, \cite{J13}, \cite{LMF16}, \cite{MS13}, \cite{MM08} and \cite{MMG07} for more information of this topic). 
      
On the other hand, if $r = 0$ the problem consists of finding directed cages or constructing regular digraphs of given girth (length of the shortest directed cycle) and minimum order. There also exists a lot of work related with this topic, as an example see \cite{ABO, BCW}; nonetheless, it quickly becomes clear that determining a good lower bound of the directed cages is not an easy problem, a construction similar to the one used for the undirected case is almost impossible to obtain due to the many possibilities that the direction of the arcs introduce. 

An upper bound for the number of vertices of directed cages is given in \cite{BCW} where the authors constructed circulant $z$-regular digraphs of girth $g$ and order $z(g-1)+1$. They also proved that, for $z \geq 1$, the lower bound of the order of a $(z,4)$-directed cage is $(5z + 4)/2$. As a consequence of this, for $1 \le z \le 3$, a $(z,4)$-directed cage has $3z + 1$ vertices. Moreover, they also prove that the order of a $(4,4)$-directed cage is $13$ and that the $(2,g)$-directed cages for $3 \le g \le 5$ have $2g-1$ vertices. Based on this results, the authors conjecture that, for any pair of integers $z \ge 1$ and $g \ge 2$, the order of a $(z,g)$-directed cage is exactly $z(g - 1) + 1$. Caccetta and Haggkvist \cite{CH78} proposed a generalization of this conjecture,
claiming that if each vertex of a digraph $D$ has out-degree at least $k$, then the girth
of $G$ is at most $\left\lfloor \frac{|V(D)|}{k} \right\rfloor $. Behzad proved both conjectures to be true for $k=2$ (see \cite{B73}); for $k=3$, the Caccetta-Haggkvist Conjecture was proved first by Bermond~\cite{B75} and later by Hamidoune \cite{H87}. Finally, Hamidoune \cite{H80,H82} proved both conjectures for $k=4$ and for all vertex-transitive digraphs. 

Inspired in the relationship between Moore Graphs and cages we propose the study of Mixed Cages motivated by the study of Mixed Moore Graphs, which have an analogous definition to the one of mixed cages but the related problem here is to find mixed graphs with maximum order and a fixed diameter. There are a lot of interesting results about  Mixed Moore graphs since they were introduced by Boz\'ak in 1979 (see \cite{B79}). 

Boz\'ak conjectured that the unique Mixed Moore graphs of diameter $k\geq 3$ are either $C_{k+1}$ or $\vec{C}_{k+1}$ (the undirected and the directed cycle respectively); this was proved by Nguyen, Miller and Gimbert in 2007 \cite{MMG07}. Moreover, they also proved that all mixed Moore graphs of diameter 2 known at that time were unique. However, this is not generally true since Jorgensen recently found  (see \cite{J13}) two non-isomorphic mixed Moore graphs of diameter 2,  out-degree 7, undirected degree 3 and order 108. As we previously mentioned, these kind of results, and other similar results related with Mixed Moore graphs (see the survey \cite{MS13} on Moore and Mixed Moore Graphs) inspired us to study and understand the behavior of Mixed Cages. 

In this paper we study the mixed cages which are closest in structure to the cages or the directed cages, in other words, we start our work taking advantage of the already existing results and ideas for graphs and directed graphs.  For this reason we study the $[1,r; g]$-mixed cages, where the directed degree is only one and the undirected degree of the graph is any integer $r\geq 1$; however as the reader will see, this "simple" change give us a lot of open questions very different
 to the undirected case. 

On the other hand, it is also natural asking what happens with the $[z,1; g]$-cages for any $z\geq 1$ and any girth $g\geq 3$ (this is, any directed degree and undirected degree equal to one).  It comes as no surprise that this problem is more difficult that the directed case, and by the analysis exposed in the introduction, it is natural that our contribution is limited, until now, to the $[z,1; 4]$-cages.

The rest of the article is organized as follows. In Section \ref{1rg} we give our results about the $[1,r; g]$-mixed cages, including a general lower bound at the beginning of the section, the calculation of the exact order of $[1,r;4]$-mixed cages and a construction of a $[1,3;5]$-mixed graph in Subsection \ref{subsec:4-5}, and the exact order of the $[1,2;g]$-mixed cages in Subsection \ref{subsec:12g}.   In Section \ref{z14} we provide upper bounds for the order of $[z,1; 4]$-mixed cages.   Our final section is dedicated to discuss future lines of work and propose open problems.

\section{Lower and upper bounds of $[1,r;g]$-mixed cages} \label{1rg}

First of all, we prove the existence of the $[z, r; g]$-mixed cage. Recall that, Erd\"os and Sachs proved in 1963 (see \cite{ES63}), the existence of the $(r,g)$-graphs. 
 
\begin{theorem}
Let $z \ge 1, r \ge 1$ and $g \ge 3$ be integers.   There exists a $[z, r; g]$-mixed
graph.
\end{theorem}

\begin{proof}
Let $D$ be a $(z,g)$-digraph, and let $G$ be an $(r,g)$-graph.   Let $H$ be the
cartesian product $G \Box D$ of mixed graphs.   Clearly, every cycle consisting
entirely of edges, and every cycle consisting entirely of arcs, has length at least
$g$, because it is contained in a $G$-fiber or in a $D$-fiber, respectively.

Let $C$ be a cycle in $H$ and consider its projection, $\pi_D (C)$, on $D$.
Clearly, $\pi_D (C)$ is a closed directed walk on $D$, and hence, it contains
a directed cycle of $D$.   Therefore, $\pi_D (C)$ uses at least $g$ vertices.
It follows that $C$ has length at least $g$.
\end{proof}

We are now ready to calculate a general lower bound for $[1,r;g]$-mixed cages.

\begin{theorem} \label{cota-inferior}
If $n[1, r; g]$ is the order of a $[1, r ; g]$-mixed cage, then $$n[1, r ;g]\ge n_0[1, r ;g]= \left\{
\begin{array}{lc} 2 \left( 1 + \sum_{i = 1}^{(g-3)/2} n_0 (r, 2i + 1) \right) + n_0 (r,g)
& \textnormal{if } g \textnormal{ is odd} \\ 2 \left( 1 + \sum_{i = 1}^{(g-2)/2} n_0
(r, 2i + 1) \right) & \textnormal{if } g \textnormal{ is even.}  \end{array} \right.$$

\end{theorem}

\begin{proof}
Recall that we introduced a tree of radius $(g-1)/2$, denoted by ${\cal{T}}_{(g-1)/2}$, to obtain the lower bound for $(r,g)$-cages, for odd $g$ (see (1) above).   We will also use it to give the lower bound of the $[1,r;g]$-cages. 

The idea of the proof is to construct mixed trees that should be contained as a subgraph of any $[1,r; g]$-cage. We have two cases, depending of the parity of $g$.   In both cases we start the construction of the tree with a directed path $(x_0, \dots ,x_{g-1})$ of length $g$.

\begin{itemize}
\item When $g\geq 3$ is odd,  we root at $x_i$ the tree ${\cal{T}}_{i}$ for $i\in \{0,\ldots,(g-1)/2\}$ and the tree ${\cal{T}}_{g-i-1}$ for $i\in \{(g+1)/2,\ldots,g-1\}$. Figure \ref{fig:cuello7} depicts the resulting mixed tree for $g=7$.

\begin{figure}[ht] 
\begin{center}
 \includegraphics[width=1\textwidth]{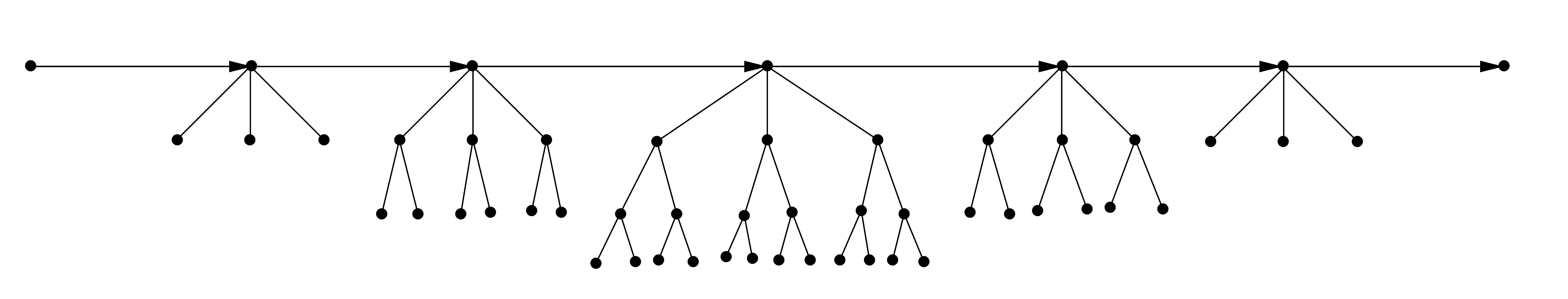}
 \caption{\small $n_0[1, 3 ;7]$}
  \label{fig:cuello7} 
\end{center}
\end{figure}

\item If $g\geq 4$ is even. We root at $x_i$ the tree ${\cal{T}}_{i}$ for $i\in \{0,\ldots,(g-2)/2\}$ and the tree ${\cal{T}}_{g-i-1}$ for $i\in \{g/2,\ldots,g-1\}$. Figure \ref{fig:cuello6} depicts the resulting mixed tree for $g=6$. 

\begin{figure}[ht] 
\begin{center}
\includegraphics[width=1\textwidth]{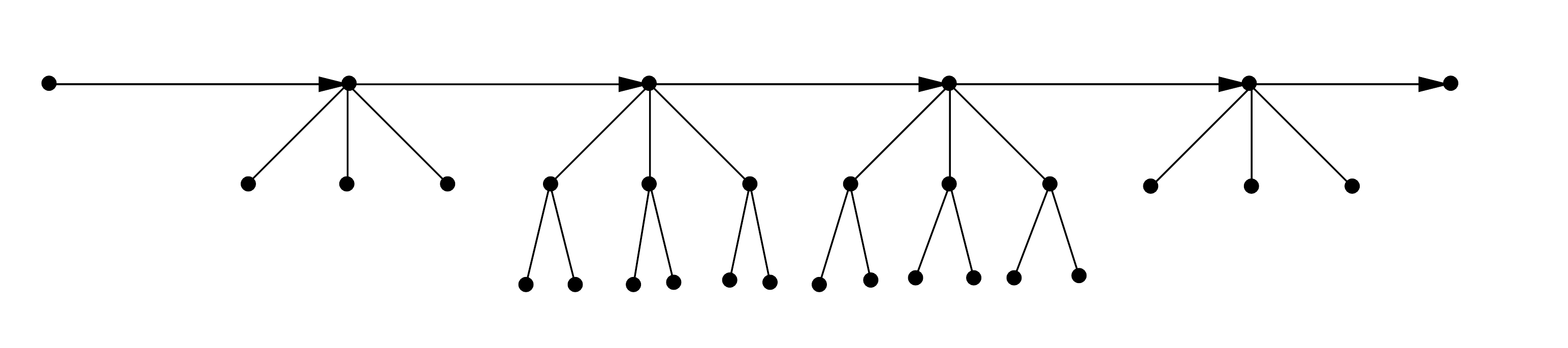}
\caption{\small $n_0[1, 3 ;6]$}
\label{fig:cuello6} 
\end{center}
\end{figure}

\end{itemize}

In both cases, it is not difficult to check that the vertices of the trees are all different, because otherwise we would immediately obtain a cycle of length smaller than $g$. Consequently, every $[1,r;g]$-cage has one of these mixed trees as subgraphs. Counting the number of vertices of the trees in each case we obtain the given lower bound directly.
\end{proof}

It is not hard to observe that one of the mixed trees constructed in the proof of Theorem \ref{cota-inferior} should be contained in any $[1,r;g]$-cage as an induced subgraph, except maybe for the edges and arcs between some leafs of the tree.   This observation will be useful later.

\subsection {The cases $g=4$ and $g=5$.} \label{subsec:4-5}

For $g=4$ we characterize all the non-isomorphic $[1, r; 4]$-mixed cages, and,  for $g=5$, we construct a family of $[1, r; 5]$-mixed graphs that provide us an upper bound for the $[1, r; 5]$-mixed graphs.

We begin analyzing the case $g = 4$.   By Theorem \ref{cota-inferior} we have that $n_0[1, r; 4]=2(r+2)$. Moreover, we have the following result:
\begin{theorem}\label{girth4}
Let $r \ge 1$ be an integer. Every  $[1, r; 4]$-cage can be obtained from
$K_{r+2, r+2}$ by taking a $2$-factor $F$ and orienting every cycle of $F$ as a directed cycle.
\end{theorem}

\begin{proof}
It follows from Theorem \ref{cota-inferior} that a $[1,r;4]$-mixed cage
has at least $2(r+2)$ vertices.   Observe that every mixed graph
obtained from $K_{r+2, r+2}$ by orienting every cycle in one of its
$2$-factors as a directed cycle, is a $[1,r;4]$-mixed graph on
$2(r+2)$ vertices.   Therefore, $[1,r;4]$-mixed cages have $2(r+2)$
vertices.

Now, let $H$ be a $[1,r;4]$-mixed cage.   As in the argument of
Theorem \ref{cota-inferior}, $H$ must contain a directed path of
length $3$, $P = (x, y, z, w)$, and the trees rooted at $y$ and $z$
are just stars.   Let $N(y)$ and $N(z)$ be $\{ y_1, \dots, y_r \}$
and $\{ z_1, \dots, z_r \}$, respectively.   

Since a $[1,r,4]$-mixed cage has precisely $2(r+2)$ vertices, we
have $V(H) = \{ x, y, z, w \} \cup N(y) \cup N(z)$, and these are
precisely the vertices of the mixed tree used in the proof of Theorem
\ref{cota-inferior}.   As we have already observed, the remaining
adjacencies of $H$ should be between the leaves of this tree.   But,
since the girth of $H$ is four, in order to fulfill the regularity restrictions
for each vertex, it follows that $x$ is adjacent to every vertex in $N(z)
\cup \{ w \}$ and $w$ is adjacent to every vertex in $N(y) \cup x$.  From
here, it is also clear that every vertex in $N(y)$ should be adjacent
to every vertex in $N(z)$.

Therefore, $(\{ x, z \} \cup N(y), \{ y, w \} \cup N(z))$ is a bipartition of
$H$.   Moreover, the underlying graph of $G$ is complete bipartite.
Hence, every $(1,r,4)$-mixed cage is obtained from $K_{r+2, r+2}$
by orienting the cycles of a $2$-factor as directed cycles.
\end{proof}

Therefore, $[1,r;4]$-cages are completely characterized by Theorem \ref{girth4}.   Figure \ref{fig:Jaulas}  depicts two non-isomorphic $[1, 2; 4]$-mixed cages, both are partial orientations of $K_{4,4}$. 

 \begin{figure}[ht] 
\begin{center}
 \includegraphics[width=.6\textwidth]{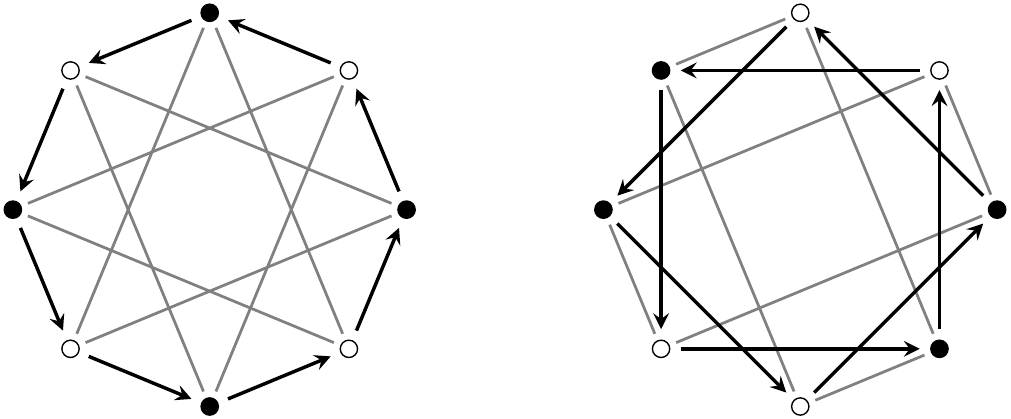}
 \caption{\small Two non-isomorphic $[1,2;4]$-mixed cages.}
  \label{fig:Jaulas} 
\end{center}
\end{figure}

Now, we study the mixed graphs of girth $5$.   Since in the following section we give a general construction for $[1, 2; g]$-mixed cages for any value of $g$, in this section we start with the study of $[1, 3; 5]$-mixed graphs.

\begin{lemma}\label{girth5}. The order of a $[1, 3; 5]$-mixed cage is bounded as follows: 
$$ 20 \leq n[1, 3; 5] \leq 28 $$
\end{lemma}

\begin{proof}
The lower bound is given by Theorem \ref{cota-inferior},  and the upper bound
is obtained by exhibiting the following $[1, 3; 5]$-mixed graph $H$ of order $28$. 

The set of vertices of $H$ is $V(H) = \{ 0, \dots, 27 \}$.   The arcs of $H$ are those
of the directed cycles $(i, i+1, \dots, i+6, i)$ for $i \in \{ 0, 7, 14, 21 \}$.   For the
edges of $H$, we give the neighborhoods of the vertices $\{ 0, \dots, 13 \}$,
the remaining adjacencies can be deducted from here.
$N(0) = \{ 14, 17, 21 \}$,
$N(1) = \{ 15, 18, 22 \}$,
$N(2) = \{ 16, 19, 23 \}$,
$N(3) = \{ 17, 20, 24 \}$,
$N(4) = \{ 14, 18, 25 \}$,
$N(5) = \{ 15, 19, 26 \}$,
$N(6) = \{ 16, 20, 27 \}$,
$N(7) = \{ 14, 23, 27 \}$,
$N(8) = \{ 15, 21, 24 \}$,
$N(9) = \{ 16, 22, 25 \}$,
$N(10) = \{ 17, 23, 26 \}$,
$N(11) = \{ 18, 24, 27 \}$,
$N(12) = \{ 19, 21, 25 \}$,
$N(13) = \{ 20, 22, 26 \}$.
The mixed graph $H$ is depicted in Figure \ref{[1,3;5]}.

It is direct to check that $H$ has the desired degree constraints.
We verified $g(H) = 5$ through the use of a computer.
\begin{figure}
\begin{center}
\includegraphics[width=.5\textwidth]{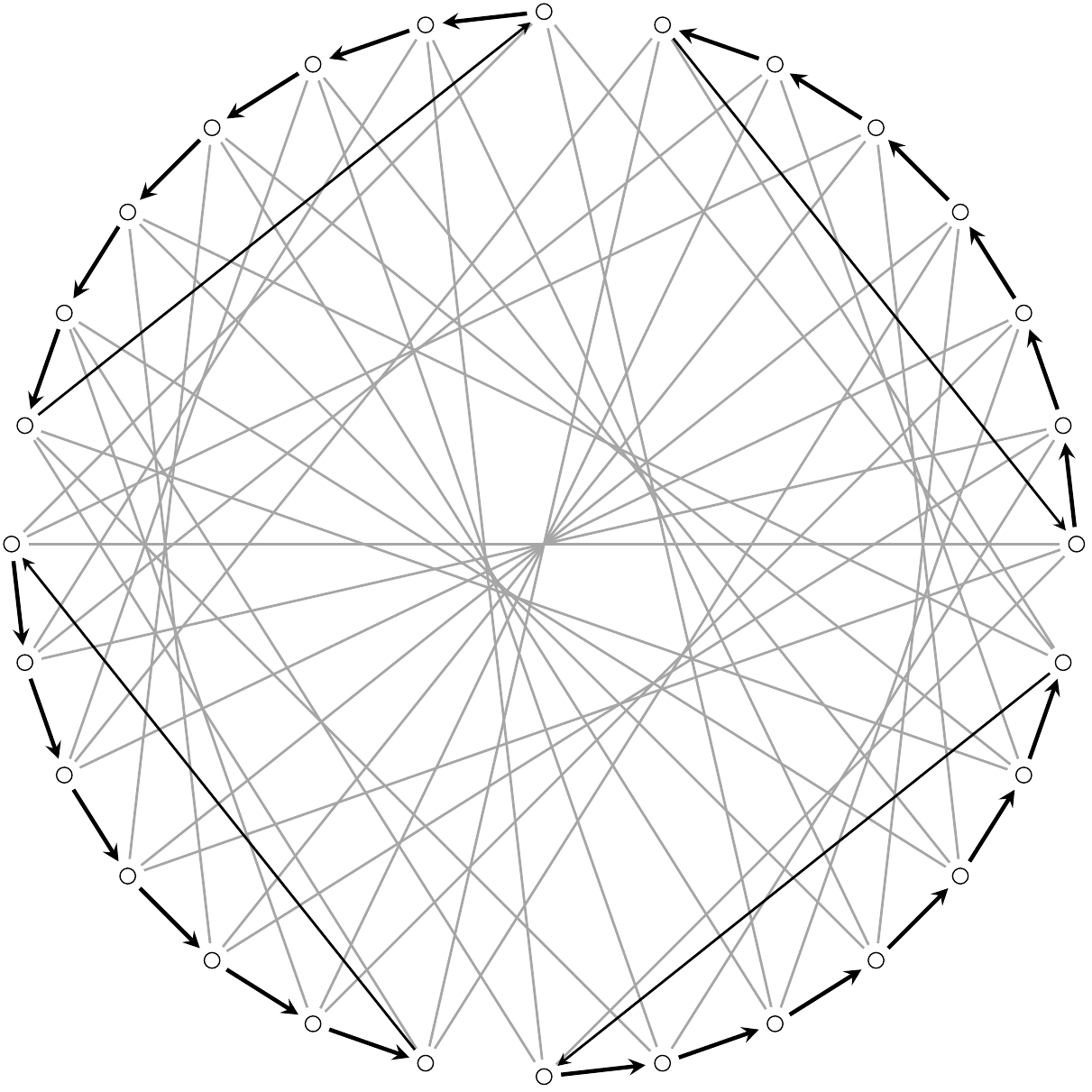}
\caption{A $[1,3;5]$-mixed graph.} \label{Figure4}
\end{center}
\end{figure}
\end{proof}

\subsection {The $[1, 2; g]$-mixed cages. } \label{subsec:12g}
Theorem \ref{cota-inferior} gives us two different bounds depending on the parity of the girth, and two different constructions to obtain these bounds. 

\begin{itemize}

\item For odd $g$, we have that:
$$n_0[1, 2; g]=2(1 + \sum_{i = 1}^{(g-3)/2} n_0 (2, 2i + 1) ) + n_0 (2,g))$$
Using the known fact that the $(2,g)$-cages are the cycles of length $g$, we have that $n_0[1, 2; g]=\frac{g^2+1}{2}$.

For each odd integer $g\geq 3$ let $C_{\frac{g^2+1}{2}}(\{g\}; \{1\})$ be the {\it circulant mixed digraph} such that  $\{x_0, \dots, x_{\frac{g^2-1}{2}} \}$ is it set of vertices; $\{x_ix_{i+g} :  0\leq i \leq  \frac{g^2-1}{2}\}$ is it set of arcs and $\{x_ix_{i+1} :  0\leq i \leq  \frac{g^2-1}{2} \}$ is it set of edges (considering the indexes modulo $\frac{g^2+1}{2}$).

Now, we have the following result:

\begin{lemma}\label{12,impar}
 For every integer $g\geq 3$  odd, the circulant mixed digraph $C_{\frac{g^2+1}{2}}(\{g\}; \{1\})$ is a $[1, 2; g]$-mixed graph of order $\frac{g^2+1}{2}$. 
\end{lemma}
\begin{proof}
By definition $C_{\frac{g^2+1}{2}}(\{g\},\{1\})$ has order $\frac{g^2+1}{2}$ and $$C=(x_0,x_g,x_{2g},\dots,x_{(\frac{g-1}{2})g},x_{(\frac{g-1}{2})g+1},\dots,x_{(\frac{g-1}{2})g+{\frac{g+1}{2}}}=x_0)$$  is a cycle of length $g$, then $C_{\frac{g^2+1}{2}}(\{g\},\{1\})$ has girth at most $g$. 

Let $C=(y_0, y_1, \dots, y_{r-1}, y_r=y_0 )$ be a cycle of minimal lenght  of $C_{\frac{g^2+1}{2}}(\{g\}; \{1\})$ with  $t$ arcs, $h$ edges and let us suppose  $r=t+h< g$.  

Without loss of generality we can suppose that  for each $i$, with $0\leq i \leq h-1$, $y_iy_{i+1} $ is an edge, and that for each $i$, with $h\leq i \leq r-1$, $y_iy_{i+1} $ is an arc.  Also, we can suppose that for each $i$, with $0\leq i \leq h$, $x_i = y_i$. Thus, $(y_h, \dots, y_r= x_0)$ is either a $y_hx_0$-directed path, or a $x_0y_h$-directed path.  If $(y_h, \dots, y_r= x_0)$ is a  $y_hx_0$-directed path it follows that $(h+gt) \cong 0$  $(mod$  $\frac{g^2+1}{2})$ and since $r=t+h< g$ we see that  $h + gt = \frac{g^2+1}{2}$. Thus $gt = \frac{g^2+1}{2} -h$ and therefore $t = \frac{g}{2} + \frac{1}{2g} - \frac{h}{g}$.  Since $h< g$ and since $\frac{g}{2} + \frac{1}{2g} - \frac{h}{g}$ is an integer it follows that $t= \frac{g-1}{2}$  and therefore $0 = \frac{1}{2} + \frac{1}{2g} - \frac{h}{g}$ which implies that $h= \frac{g+1}{2}$ but then $g> r = t+h = g$ which is a contradiction.  If $(y_h, \dots, y_r= x_0)$ is a $x_0y_h$-directed path, since $h< g$ it follows that $t\geq 2$ and $tg \cong h$ $(mod$  $\frac{g^2+1}{2})$, and since $t< g$ we see that $gt = \frac{g^2+1}{2} + h$.  Thus $t = \frac{g}{2} + \frac{1}{2g} + \frac{h}{g}$  and therefore,  since $h< g$ and since $\frac{g}{2} + \frac{1}{2g} + \frac{h}{g}$ is an integer, we see that $t =  \frac{g+1}{2}$ and $h = \frac{g-1}{2}$ which implies that  
$g> r = t+h = g$ which is a contradiction.\end{proof}

\item For even $g$ we have that: 

$$n_0[1, 2; g]=2 \left( 1 + \sum_{i = 1}^{(g-2)/2} n_0
(2, 2i + 1) \right)$$

Again, as the $(2,g)$-cages are the cycles of length $g$, we have that, in this case, $n_0[1, 2; g]=\frac{g^2}{2}$.

 Now, we define a graph $H_g$ taking first a mixed graph $G_g=C_g\Box  \overrightarrow{P_{g/2}}$ which is the cartesian product of an undirected cycle $C_g=(0, \dots, g-1,0)$ of length $g$ with a directed path $\overrightarrow{P_{g/2}}=(0, \dots, g/2 - 1)$ of length $\frac{g}{2}$, and adding to $G_g$ the set of arcs $A=\{((i, g/2 - 1), (i + g/2, 0)),\  \mbox{for} \ 0\leq i\leq g-1 \}$ (mod $g$) to obtain $H_g$.

Now, we will prove that: 
\begin{lemma}\label{12,par}
The mixed graph $H_g$ is a $[1,2; g]$-mixed graph of order $\frac{g^2}{2}$.
\end{lemma}

\begin{proof}
By definition $H_g$ is a $1$-arc-regular, $2$-edge-regular mixed graph of order
$\frac{g^2}{2}$. We will verify that the girth of $H_g$ is $g$.

Note that every cycle consisting entirely of edges or entirely of arcs has length
exactly $g$.  Consider $C$ a cycle using at least one edge and at least one arc.
Since $H_g$ is clearly vertex-transitive, we will assume without loss of generality
that $(0,0)$ is the first (and last) vertex of $C$.   Assume that $H_g$ is drawn in the
plane in a grid-like fashion, with $P_g$ in the $X$-axis and $\overrightarrow{P_{g/2}}$
in the $Y$-axis.

The only downward arcs are in $A$, hence, at least one of such arcs should be used by $C$.   
Note that if two of these arcs are used in $C$, then at least $g - 2$ upward arcs are used
by $C$, and hence, its length is at least $g$.  So, suppose that the only
downward arc used by $C$ is $((i, g/2 - 1), (i + g/2, 0))$ (mod $g$) for some fixed
$0 \leq i \leq g-1$. Hence, $C$ uses exactly $g/2$ arcs.

We will consider the case when $i \le g/2 - 1$, the remaining case can be dealt
similarly.   At least $i$ edges are used to reach $(i, g/2 - 1)$ from $(0,0)$, and
at least $g/2 - i$ edges are used to reach $(0,0)$ from $(i + g/2, 0)$.   Hence,
$C$ uses at least $i + g/2 - i = g/2$ edges.   Together with the $g/2$ arcs, the
length of $C$ is at least $g$.
\end{proof}

\end{itemize}

 From the  previous results  we have the following theorem:

\begin{theorem} For any integer $g\geq 3$, the $[1, 2; g]$-mixed cages have
order $\frac{g^2 + 1}{2}$ if $g$ is odd, and $\frac{g}{2}$ if $g$ is even.
\end{theorem}

\section{Results of $[z,1;4]$-mixed cages} \label{z14}

In this section we study the case of mixed graphs with any directed regular degree and undirected degree equal to $1$ (that is we have a ``directed graph" with an ``undirected matching"), and girth equal to four.   As we explained in the introduction, the study of directed cages is complicated, even for girth equal to four, this is the reason we start this introductory work with this seemingly simple case. 

First we give a family of $[z, 1; 4]$-mixed graphs that give us a lower bound of $n[z, 1; 4]$, also we will prove that they are mixed cages for $z=1,2$. 

\begin{lemma} \label{cotasup(z,1,4)}
Let $n$ be a positive integer.
\begin{enumerate}
	\item If $n$ is odd, then the mixed graph obtained from the circulant digraph  $$C_{3(n+1)} \left(1, \dots,
		\frac{n+1}{2}, \frac{3(n+1)}{2}+1, \dots, 2n+1 \right),$$ by adding edges between antipodal
		 vertices, has girth $4$ and is $n$-arc-regular.
	\item If $n$ is even, then the mixed graph obtained from the circulant digraph  $$C_{3n+2} \left(1, \dots,
	 	\frac{n}{2}, \frac{3n+2}{2}+1, \dots, 2n+1 \right),$$ by adding edges between antipodal
		vertices, has girth $4$ and is $n$-arc-regular.
\end{enumerate}
\end{lemma}

\begin{proof}
Let $D_n$ be the digraph described in the theorem.   In either case, $D_n$ is a circulant mixed graph, hence it is vertex-transitive.   So, it suffices to show that any of its vertices does not belong to a mixed triangle.   Suppose that $V(D_n) = \{ 0, \dots, 3(n+1) \}$ when $n$ is odd and $V(D_n) = \{ 0 \dots, 3n+2 \}$ for $n$ even.   We observe that the vertex $0$ does not belong to any mixed triangle.

Assume first that $n$ is even.   It is not hard to calculate the following sets.
$$N^+(N^+(0)) = [2,n] \cup \left[ \frac{3n+2}{2}+2, \frac{5n+2}{2} \right],$$
$$N^-(0) = \left[ n+1, \frac{3n+2}{2}-1 \right] \cup \left[ \frac{5n+4}{2}, 3n+1 \right],$$
$$N^+ \left( \frac{3n+2}{2} \right) = \left[ 1, \frac{n}{2} \right] \cup \left[ \frac{3n+2}{2}+1, 2n+1 \right].$$
Since the intersection between the first and second sets is empty, there are no directed triangles containing $0$.   There are not mixed triangles using two arcs and one edge because the intersection of the second and third sets is empty.

 For $n$ odd  we have the following equalities.
$$N^+(N^+(0)) = [2,n+1] \cup \left[ \frac{3(n+1)}{2}+2, \frac{5n+3}{2} \right],$$
$$N^-(0) = \left[ n+2, \frac{3(n+1)}{2}-1 \right] \cup \left[ \frac{5n+5}{2}, 3n+2 \right],$$
$$N^+ \left( \frac{3(n+1)}{2} \right) = \left[ 1, \frac{n-1}{2} \right] \cup \left[ \frac{3(n+1)}{2}+1, 2n+2 \right].$$ 
Analogous arguments as those on the first case complete the proof.
\end{proof}

The following conjecture was proved for the case $r=2$ (see \cite{B73}):
\begin{conjecture} \label{CH-2}(\cite{CH78})
Every digraph on $n$ vertices with minimum out-degree at least $r$ has a cycle of length
$\lceil \frac{n}{r} \rceil$.
\end{conjecture}

So, the following result is now easy to derive. 

\begin{theorem} \label{z14}
Let $z \ge 1$ be an integer, then we have that: 

$$n [z, 1; 4] \le \left\{ \begin{array}{cr}
3(z+1) & \textnormal{ if } z \textnormal{ is odd.} \\ 3z + 2 & \textnormal{ otherwise.} \end{array}
\right.$$   Moreover, the equality holds for $z \in \{ 1, 2 \}$.
\end{theorem}

\begin{proof}
The inequality follows directly from Lemma \ref{cotasup(z,1,4)}.  The equality for $z = 1$ is
trivial to verify.   The equality for $z = 2$ follows from the proved case of Conjecture \ref{CH-2}.
\end{proof}

\section{Conclusions and Future Work}

The objective of this paper is to introduce a new research direction: the mixed graph analogue to cages in undirected graphs.   As we mentioned in the introduction, this was inspired by the numerous interesting results that have been recently obtained in the study of "Mixed Moore Graphs" (the mixed graph analogue to Moore Graphs).

As many papers introducing a new subject, in this work we give the basics of this topic and study some simple cases.   We successfully introduce the $[z,r;g]$-mixed cages, proving their existence for any suitable value of $z, r$ and $g$; also, some upper and lower bounds on the order of a $[z,r;g]$-mixed cage are obtained for particular values of these parameters.   In some cases, we obtain the exact order of a $[z,r;g]$-mixed cage, and the precise family of $[z,r;g]$-mixed cages.

We believe that there is plenty of work to do in this new field.   Our first objective is continue with the case $z=1$, because $[1,r;g]$-mixed graphs are the mixed graphs ``closest'' in structure to graphs, and the results for this case could be inspired by known results on (undirected) cages.  In this direction, the obvious first step is to improve, for any $r \ge 3$, the bound given in this paper for the order of  $[1,r;5]$-mixed cages. Furthermore, we propose to study $[1,r;6]$-mixed cages. It is easy to note that, each one of these two cases is interesting in itself, but also the relation between them is interesting; it is important to recall that, for cages, there are exactly four values of $r$ for which a Moore cage can exist: $2,3,7$ and $57$, and exactly three Moore cages of girth five are known (for the values $2,3,7$) whereas, for $g=6$, there exists an infinite number of them (the incidence graph of a projective plane of order $r-1$ is a $(r,6)$-Moore cage, whenever $r-1$ is a prime power). So, inspired by the previous results on cages some natural questions arise: Is it true that there are finitely many values of $r$ such that there exist $[1,r;5]$-mixed cages that attain the lower bound given in this paper?   If none exists: It is possible improve this bound?   Furthermore, are there finitely many values of $r$ such that there exist $[1,r;6]$-mixed cages that attain the lower bound given in this paper?

Finally, as we mentioned in the introduction, as the value of $z$ grows, the study of $[z,r;g]$-mixed cages becomes increasingly similar to the study of directed cages, which has proved to be a very difficult subject.   As a first problem, we propose to find a lower bound for the order of a $[z,1;4]$-mixed cage.   This, together with the upper bound given in Theorem \ref{z14}, could result in an exact value for the order of this family of mixed cages.   Of course, it would be interesting to obtain bounds for $g = 4$ and any value of $z$ and $r$.

\subsection*{Acknowledgment}
The authors wish to thank the anonymous referees of this paper.

{ Research   supported by CONACyT-M\'exico under projects 178395, 166306, and PAPIIT-M\'exico under project IN104915}. 

\end{document}